\documentclass[12pt,reqno]{amsart}
\usepackage{amssymb}
\usepackage{mathtools, amsmath, amsthm, thmtools}
\usepackage[breaklinks]{hyperref}
\theoremstyle{thmrm} 
\usepackage{cleveref}

\usepackage{graphicx}

\usepackage{longtable}
\theoremstyle{plain}
\newtheorem{theorem}{Theorem}[section]
\newtheorem{lemma}[theorem]{Lemma}

\newtheorem{proposition}[theorem]{Proposition}
\newtheorem{remark}[theorem]{Remark}
\theoremstyle{proof}
\theoremstyle{definition}

\newtheorem{definition}[theorem]{Definition}

\theoremstyle{remark}
\theoremstyle{lamma}

\numberwithin{equation}{section}
\numberwithin{theorem}{section}

\crefname{lemma}{Lemma}{Lemmas}
\crefname{theorem}{Theorem}{Theorems}
\crefname{proposition}{Proposition}{Propositions}
\crefname{corollary}{Corollary}{Corollaries}
\crefname{remark}{Remark}{Remarks}

\title[]{Integrable representations for toroidal Lie algebras co-ordinated by rational quantum torus} 
\author{Suman Rani, Punita Batra\textsuperscript{*}}
\address{Harish-Chandra Research Institute,  A CI of Homi Bhabha National
Institute, Chhatnag Road, Jhunsi, Prayagraj - 211019}
\email{sumanrani@hri.res.in, batra@hri.res.in}

\begin{document}

\maketitle
\let\thefootnote\relax\footnotetext{* Corresponding author}
\begin{abstract}
   We classify irreducible integrable modules with finite-dimensional weight spaces for toroidal Lie algebras coordinated by rational quantum torus with trivial central action. Let $\mathbb{C}_q$ denote the rational quantum torus associated with a rational quantum matrix $q$, and let $\hat{\tau}(d,q)$ be the toroidal Lie algebra coordinated by rational quantum torus obtained by adjoining the derivation space $D$ to the universal central extension $\tilde{\tau}(d,q)=\mathfrak{sl}_d(\mathbb{C}_q)\oplus HC_1(\mathbb{C}_q)$ of $\mathfrak{sl}_d(\mathbb{C}_q)$. The case of nontrivial central action was previously classified by S. Eswara Rao and K. Zhao \cite{[18]}. The present work completes the classification by describing all irreducible integrable $\hat{\tau}(d,q)$-modules with finite-dimensional weight spaces in the case where the $n$-dimensional center $C$ acts trivially on the modules.
\end{abstract}
\vspace{1em}
\section{Introduction}
Affine Lie algebras play a central role in both mathematics and theoretical physics, arising naturally in the study of two-dimensional conformal field theory, string theory, and statistical mechanics. These algebras admit a rich representation theory and serve as a starting point for understanding more general multivariable analogues. In particular, toroidal Lie algebras are natural $n$-variable generalizations of affine Lie algebras. They can be realized as the universal central extensions of multiloop algebras~\cite{[1],[9]} and were first introduced by Moody, Rao, and Yokonuma~\cite{[16]}. Over the past three decades, toroidal Lie algebras have attracted considerable attention due to their deep connections with extended affine Lie algebras (EALAs), which have become a central object of study in the past twenty years. Notably, toroidal Lie algebras form the cores of EALAs, and those of type $A_{d-1}$ are associated with the Lie algebra $\mathfrak{gl}_d(\mathbb{C}_q)$, where $\mathbb{C}_q$ is the Laurent polynomial algebra in $n$ noncommuting variables $t_1, \dots, t_n$ defined by a rational quantum matrix $q$ whose entries are roots of unity.

\medskip

Integrable representations of affine and toroidal Lie algebras have played a foundational role in the development of the theory, beginning with the classical work of Kac~\cite{[8]}, the influential results of Chari~\cite{[4]} and continuing to the present. Such representations are of particular importance because they admit a highest-weight theory, possess finite-dimensional weight spaces, and are central to applications in mathematical physics and combinatorics.
Parallel developments have also occurred in the context of \emph{Lie superalgebras}, which generalize Lie algebras through a $\mathbb{Z}_2$–grading into even and odd parts. Integrable modules for affine Lie superalgebras have been extensively studied, notably in the work of Futorny and Rao~\cite{[15]}, where a complete classification of irreducible integrable modules was obtained. These results reveal close structural analogies between the affine and toroidal settings, motivating further study of integrable representations in the superalgebraic setting.

\medskip

The structure of the paper is as follows. In Section~2, we define the rational quantum torus $\mathbb{C}_q$ associated with the rational quantum matrix $q$. We denote by $\mathfrak{gl}_d(\mathbb{C}_q)$ the Lie algebra of $M_d(\mathbb{C}) \otimes \mathbb{C}_q$, and consider its Lie subalgebra $\tau(d,q)=\mathfrak{sl}_d(\mathbb{C}_q)$, where $d \ge 2$ and $n \ge 2$ are integers. The trace of matrices in $\mathfrak{sl}_d(\mathbb{C}_q)$ lies in $[\mathbb{C}_q,\mathbb{C}_q]$. The universal central extension of $\tau(d,q)$ is $\tilde{\tau}(d,q) = \tau(d,q) \oplus HC_1(\mathbb{C}_q)$. Adding the linear span $D$ of the degree derivations $d_1, \dots, d_n$ to $\tilde{\tau}(d,q)$ yields the toroidal Lie algebra $\hat{\tau}(d,q) = \tilde{\tau}(d,q) \oplus D$. We classify irreducible integrable modules $V$ for $\hat{\tau}(d,q)$ with finite-dimensional weight spaces in which the $n$-dimensional center $C$ acts trivially. The classification of irreducible integrable $\hat{\tau}(d,q)$-modules with nontrivial central action was obtained by S.~Eswara Rao and K.~Zhao in~\cite{[18]}. Our work thus completes the classification in the case of trivial center action.

\medskip

In Section~3, we introduce the notion of integrable modules for $\hat{\tau}(d,q)$ and develop their basic structural properties. We describe the weight lattice, define real and null roots and construct the associated Weyl group. Several Lemmas paralleling the classical affine case provide the technical foundation for subsequent sections. Section~4 focuses on the action of the central extension $HC_1(\mathbb{C}_q)$ on irreducible integrable modules. We prove that if each central element $c_i=\langle t_i,t_i^{-1}\rangle$ acts trivially on such a module $V$, then the entire central part $HC_1(\mathbb{C}_q)$ acts trivially, thereby reducing the classification problem to modules over the quotient algebra where the center acts by zero.

Section~5, forming the core of the paper, develops the \emph{highest weight theory} for irreducible integrable modules of $\hat{\tau}(d,q)$. Throughout this section, we work with $\mathfrak{gl}_d(\mathbb{C}_q)$-modules in place of $\tau(d,q)$-modules, using the decomposition (see~\cite[1.10]{[18]})
$
\mathfrak{gl}_d(\mathbb{C}_q)
= \mathfrak{sl}_d(\mathbb{C}_q) \oplus (I \otimes Z(\mathbb{C}_q)),
$
where the kernel is central, so that each $\mathfrak{sl}_d(\mathbb{C}_q)$-module may be viewed as a $\mathfrak{gl}_d(\mathbb{C}_q)$-module with the additional central part acting trivially. We introduce the \emph{highest weight space}
$
V^{+}
$
and analyze its structure as a module over $H \otimes \mathbb{C}_q \oplus D$. It is shown that $V^{+}$ is nonzero and irreducible under this subalgebra.
We then construct a family of commuting bijective linear operators, referred to as \emph{highest central operators}, acting on $V^{+}$ and encoding the influence of the central elements of the quantum torus. These operators provide a bridge between the graded representation theory of $\hat{\tau}(d,q)$ and the non-graded theory for $\tau(d,q)$. Using these constructions, we obtain a finite-dimensional quotient of $V$ that carries a natural $\mathfrak{gl}_d(\mathbb{C}_q)$-module structure. The Section proves an important result, \cref{theorem 5.16}, which establishes the existence of a \emph{finite-dimensional irreducible quotient} of $V$ as a module over $\mathfrak{gl}_d(\mathbb{C}_q)$.

In Section~6, we reverse this process: starting from a finite-dimensional irreducible $\mathfrak{gl}_d(\mathbb{C}_q)$-module, we construct the corresponding graded $\hat{\tau}(d,q)$-module by tensoring with the Laurent polynomial ring $A_n = \mathbb{C}[t_1^{\pm1}, \ldots, t_n^{\pm1}]$. We show that the resulting module is completely reducible, integrable, and has finite-dimensional weight spaces, with only finitely many irreducible components up to grade shift. This yields a precise correspondence between graded and non-graded modules in the trivial central action case. Finally, Section~7 recalls the classification of finite-dimensional $\tau(d,q)$-modules obtained by S.~Eswara Rao and K.~Zhao~\cite{[18]} and applies it to our setting. Combining these results with our reduction theorems yields a complete classification of irreducible integrable $\hat{\tau}(d,q)$-modules with finite-dimensional weight spaces under trivial central action, completing the overall classification program for integrable representations of toroidal Lie algebras coordinated by rational quantum tori.

\section{notations and preliminaries}\label{section 2}
First we recall some notations and known results from \cite{[3]}. We will always denote $\mathbb{Z},\;\mathbb{Z_+},\;\mathbb{N} $ and $\mathbb{C}$ the set of integers, non-negative integers, positive integers and complex numbers, respectively.  Let $q = (q_{ij})_{1\leq i,j\leq n}$ where $q_{ij}$'s are complex numbers that are nonzero and $q_{ii}=1$, $q_{ij}=q_{ji}^{-1}, \; \forall \; 1 \le i,j \le n$. The quantum torus associated with this $q$ is $\mathbb{C}_q= \mathbb{C}[t_1^{\pm},\cdots ,t_n^{\pm}]$; the non-commutative laurent polynomial algebra with the relation defined as $t_it_j = q_{ij}t_jt_i$. Note that $\mathbb{C}_q$ defined above is $\mathbb{Z}^n$-graded with one-dimensional graded components. We assume a rational quantum torus i.e. $q_{ij}^{N_{ij}}=1$ for some positive integer $N_{ij}, \; \forall \; 1 \le i,j \le n$.\\
For $a=(a_1,\cdots ,a_n) \in \mathbb{Z}^n$, we write $t^a = t_1^{a_1} \cdots t_n^{a_n} \in \mathbb{C}_q$. Now we define the maps, $\sigma,f : \mathbb{Z}^n \times \mathbb{Z}^n \to \mathbb{C}^\times$ by 
\begin{equation}
    \sigma(a,b) = \prod_{\substack{1\leq i \leq j \leq n}}q_{ij}^{a_jb_i} ,
\end{equation}
\begin{equation}
    f(a,b) = \sigma(a,b) \sigma(b,a)^{-1}.
\end{equation}
For ${a,b} \in \mathbb{Z}^n,~ k \in \mathbb{Z}$, it is easy to check the following
\begin{equation}
    f(a,b) = \prod_{i,j=1}^n q_{ji}^{a_jb_i}, 
\end{equation}
$$f(a,b) = f(b,a)^{-1},$$
$$\sigma(a+b,c) = \sigma(a,c) \sigma(b,c),$$
$$\sigma(a,b+c) = \sigma(a,b) \sigma(a,c),$$
$$f(a+b,c) = f(a,c)f(b,c),$$
$$f(a,b+c) = f(a,b)f(a,c),$$
$$f(ka,a) = f(a,ka) = 1,$$
\begin{equation}
    t^at^b = \sigma(a,b)t^{a+b}, \hspace{1cm} t^at^b = f(a,b)t^bt^a.
\end{equation}
The \emph{radical} of $f$ is given by
\begin{equation}
    \operatorname{rad}f  = \{a \in \mathbb{Z}^n \mid f(a,b) = 1, \; \forall b \in \mathbb{Z}^n\}.
\end{equation}
Note that $m \in \operatorname{rad}f $ $\Leftrightarrow$ $f(a,b) = 1,$ $\forall a,b\in \mathbb{Z}^n $ with $a+b = m$.
Further $\operatorname{rad}f $ is a subgroup of $\mathbb{Z}^n$.

\begin{proposition} {\normalfont(\cite[Proposition 2.44]{[3]})}\hfill\break
    Let $\mathbb{C}_q$ be the rational quantum torus defined as above, then:
    \begin{enumerate}
        \item The center $Z(\mathbb{C}_q)$ of $\mathbb{C}_q$ has a basis consisting of monomials $t^a$ with $a \in \operatorname{rad}f $.
        \item The Lie subalgebra $[\mathbb{C}_q, \mathbb{C}_q]$ of $\mathbb{C}_q$ has a basis consisting of monomials $t^a$ with $a \in \mathbb{Z}^n \setminus \operatorname{rad} f $.
        \item $\mathbb{C}_q = [\mathbb{C}_q,\mathbb{C}_q] \oplus Z(\mathbb{C}_q)$.
    \end{enumerate}
\end{proposition}

Let $M_d(\mathbb{C})$ denote the associative algebra of $d\times d$ complex matrices with standard basis elements $E_{ij}$ for $1 \le i,j \le n$. We denote the corresponding Lie algebra by $\mathfrak{gl}_d(\mathbb{C})$ and the simple subalgebra of $\mathfrak{gl}_d(\mathbb{C})$ of trace-zero matrices by  $\mathfrak{sl}_d(\mathbb{C})$.

Further we denote the associative matrix algebra of $d \times d$ matrices with entries in $\mathbb{C}_q$ by $M_d(\mathbb{C}_q)$.
$$M_d(\mathbb{C}_q) \cong M_d(\mathbb{C}) \otimes \mathbb{C}_q$$ and the corresponding Lie algebra is denoted by $\mathfrak{gl}_d(\mathbb{C}_q)$ and the Lie bracket is defined by
$$[X \otimes t^a, Y \otimes t^b]_0 = XY \otimes t^at^b - YX\otimes t^bt^a.$$
We define $\mathfrak{sl}_d(\mathbb{C}_q) = \{X \in M_d(\mathbb{C}_q) \; | \; Tr(X) \in [\mathbb{C}_q, \mathbb{C}_q]\}$.\\
Now we consider the following Lie subalgebra inside $M_d (\mathbb{C}_q)$
\[
\tau(d,q) = (I_d \otimes [\mathbb{C}_q,\mathbb{C}_q]) \oplus (\mathfrak{sl}_d(\mathbb{C}) \otimes \mathbb{C}_q),
\]
where the Lie bracket is given as 
$$[I_d \otimes [t^a,t^b],X \otimes t^c] = X \otimes [[t^a,t^b],t^c],$$
$$[I_d \otimes [t^a,t^b],I_d \otimes [t^c,t^d]] = I_d \otimes [[t^a,t^b],[t^c,t^d]],$$
$$[X \otimes t^a, Y \otimes t^b] = [X,Y]' \otimes \frac{t^a \circ t^b}{2} + (X \circ Y) \otimes \frac{1}{2}[t^a,t^b] + \frac{1}{d} Tr(XY)I_d \otimes [t^a,t^b],$$
where $$t^a \circ t^b = t^at^b + t^bt^a,$$
$$[X,Y]' = XY - YX,$$
$$X \circ Y = XY + YX - \frac{2}{d}Tr(XY)I_d,$$
$$[t^a,t^b] = t^at^b - t^bt^a.$$
Now it is easy to see that (also mentioned in \cite{[3]})
\[
\tau(d,q)=\mathfrak{sl}_d(\mathbb{C}_q).
\]

We recall the universal central extension of $\tau(d,q)$ from \cite{[3]}. Let \(J\subset\mathbb{C}_q\otimes\mathbb{C}_q\) be the subspace spanned by
\[
x\otimes y+y\otimes x,\quad xy\otimes z+yz\otimes x+zx\otimes y,
\quad x,y,z\in\mathbb{C}_q.
\]
For \(x,y\in\mathbb{C}_q\), let \(\langle x,y\rangle_0\) be the image of \(x\otimes y+J\) in
\((\mathbb{C}_q\otimes\mathbb{C}_q)/J\).

We define
\[
HC_1(\mathbb{C}_q)
= \{\langle t^a,t^b\rangle\mid a+b\in\operatorname{rad}f\},
\quad
\langle t^a,t^b\rangle
= \delta_{a+b,\operatorname{rad}f}\langle t^a,t^b\rangle_0.
\]
Let $\tilde{\tau}(d,q) = \tau(d,q) \oplus HC_1(\mathbb{C}_q)$ with the Lie brackets 
\[
[X \otimes t^a, Y \otimes t^b] = [X \otimes t^a, Y \otimes t^b]_0 + Tr(XY) \langle t^a,t^b \rangle_0 \delta_{a+b,\operatorname{rad}f}
\]
where \[
\delta_{a, \operatorname{rad}f}= \begin{cases}
    1 & \text{if}\;a \in \operatorname{rad}f,\\
    0 & \text{if}\;a \notin \operatorname{rad}f.
\end{cases}
\]
Then from \cite{[3]}, $\tilde{\tau}(d,q)$ is the universal central extension of $\tau (d,q).$
We note that $\tilde{\tau}(d,q)$ is $\mathbb{Z}^n$-graded and to see this fact we add degree derivations. Let $D$ be the linear span of degree derivations $d_1, \cdots ,d_n$.\\
Let $\hat{\tau}(d,q) = \tilde{\tau}(d,q) \oplus D$ and Lie brackets are extended as follows:
$$[d_i,X \otimes t^a] = a_iX \otimes t^a,$$
$$[d_i,I\otimes t^a] = a_iI\otimes t^a,$$
$$[d_i,\langle t^a,t^b\rangle] = (a_i+b_i)\langle t^a,t^b\rangle,$$
$$[d_i,d_j] = 0.$$
Then the \emph{center} $C $ of $\hat{\tau}(d,q)$ is spanned by $c_i = \langle t_i,t_i^{-1}\rangle $ for $i = 1, \cdots ,n$.

\begin{lemma} \label{lemma 2.2}
{\normalfont(\cite[Lemmas 3.18, 3.19]{[3]})}\hfill \break
For $\langle\;,\; \rangle$ and $HC_1(\mathbb{C}_q)$ as defined above, we have the following
    \begin{enumerate}
    \item $\langle 1, t^a \rangle = 0 \quad \forall a \in \mathbb{Z}^n.$
    \item $\displaystyle \langle t^a, t^b (t^a)^{-1} \rangle = \sum_{i=1}^d a_i \langle t_i, t^b t_i^{-1} \rangle.$
    \item $\displaystyle \langle t^a, t^b \rangle = \sigma(a,b) \sum_{i=1}^d a_i \langle t_i, t^{a+b} t_i^{-1} \rangle.$
    \item $(t^b)^{-1} = \sigma(b,b) t^{-b}.$
    \item $\dim HC_1(\mathbb{C}_q)_a =
        \begin{cases}
            0    &  a\notin \operatorname{rad}f,\\
            n-1  &  a \in \operatorname{rad}f \setminus\{0\},\\
            n    &  a=0.
        \end{cases}$\\
    \end{enumerate}
\end{lemma}

In the subsequent sections, we classify \emph{irreducible integrable modules} for
the Lie algebra $\hat{\tau}(d,q)$ with finite-dimensional weight spaces and trivial center action.
A module is called \emph{graded} if it is a $\hat{\tau}(d,q)$-module and
\emph{non-graded} if it is a module over the quotient of $\hat{\tau}(d,q)$ by its derivation subalgebra.

\section{Integrable modules}

In this section, we introduce the notion of integrable modules for the Lie algebra
$\hat{\tau}(d,q)$ and establish their basic properties.  
Throughout we assume that $V$ is a $\hat{\tau}(d,q)$-module with finite-dimensional weight spaces.

\begin{definition}
    Let $r=(r_1, \cdots ,r_n) \in \mathbb{Z}^n$. A linear map $z : V \to V$ is called a \emph{central operator} of degree $r$ if \begin{enumerate}
        \item $z$ commutes with the action of $\tilde{\tau}(d,q)$,
        \item $d_iz-zd_i=r_iz$ for all $i= 1, \cdots ,n$.
    \end{enumerate}
\end{definition}

The following result adapted from \cite[Proposition 2.4]{[18]} describes the structure of central operators acting on irreducible modules.

\begin{proposition}\label{proposition 2.4}
    Let $V$ be an irreducible $\hat{\tau}(d,q)$ module with finite-dimensional weight spaces. Then after a suitable co-ordinate change for $\mathbb{C}_q$, the following hold
    \begin{enumerate}
        \item There exist a non-negative integer $k$ and non zero central operators $z_1, \cdots z_k$ of degree $(m_1,0, \cdots ,0), \cdots (0,\cdots ,m_k, \cdots ,0)$ for some positive integers $m_1, \cdots, m_k$.
        \item If $\mathbb{C}_q$ is rational, then $k<n$.
        \item If $\langle t_i,t^rt_i^{-1}\rangle \neq 0$ on $V$, then $i\geq k+1$ and $r_{k+1} = \cdots = r_n = 0$.
    \end{enumerate}
\end{proposition}

We now introduce the notions of \emph{real roots}, \emph{null roots}, and the \emph{Weyl group} associated with the Lie algebra $\hat{\tau}(d,q)$. It is worth emphasizing that, in this context, the Weyl group coincides with that arising in the commuting torus case. This equivalence follows from the observation that the set of weights of any irreducible integrable module with finite-dimensional weight spaces exhibits the same structural properties as in the commuting case. Consequently we adopt the notation of Section~1 in \cite{[12]}.

We recall that the Cartan subalgebra of $\mathfrak{sl}_d(\mathbb{C})$ is $\dot{\mathfrak{h}}=\{\text{trace-zero diagonal matrices in} \; \mathfrak{gl}_d(\mathbb{C})\}$. Let $\alpha_1, \cdots, \alpha_{d-1} \in \dot{\mathfrak{h}}^*$ be the standard simple roots and let $\alpha_1^{\vee}, \cdots ,\alpha_{d-1}^{\vee} \in \dot{\mathfrak{h}}$ be the coroots of the finite-dimensional Lie algebra $\mathfrak{sl}_d(\mathbb{C})$.

We denote by
\[
\dot{\Delta} \subset \dot{\mathfrak{h}}^*, \quad \dot{Q}=\bigoplus_{i=1}^{d-1}\mathbb{Z}\alpha_i, \quad \dot{Q}^+=\bigoplus_{i=1}^{d-1}\mathbb{N}\alpha_i
\]
the corresponding root system, root lattice and positive root lattice, respectively. Let $\mathcal{W}_0$ be the Weyl group of $\mathfrak{sl}_d(\mathbb{C})$, generated by reflections $r_{\alpha_i}$.

We now define \[\mathfrak{h} = \dot{\mathfrak{h}} \oplus \sum_{i=1}^n \mathbb{C}c_i \oplus \sum_{i=1}^n \mathbb{C}d_i\] where $c_i = \langle t_i,t_i^{-1}\rangle$ are the central elements and $d_i$ are degree derivations introduced earlier. Hence $\dim \mathfrak{h} = 2n+d-1$. 
We now define the elements $\delta_i, \omega_i \in \mathfrak{h}^*$, $1\leq i \leq n$ by
\[
\delta_i(\dot{\mathfrak{h})} = 0, \quad \delta_i(d_j) = \delta_{ij}, \quad \delta_i(c_j) = 0,
\]
\[
\omega_i(\dot{\mathfrak{h}})=0, \quad \omega_i(d_j)=0, \quad \omega_i(c_j) = \delta_{ij}.
\] 
Then $\{\alpha_i,\delta_j,\omega_j \mid 1\leq i \leq d-1, \; 1 \leq j \leq n\}$ forms a basis of $\mathfrak{h}^*$.

Let $(\;,\;)$ be the standard symmetric bilinear form on $\dot{\mathfrak{h}}^*$, normalized by $(\alpha_i,\alpha_i)=2$. We extend it to $\mathfrak{h}^*$ by
\[
(\dot{\mathfrak{h}}^*,\delta_i) = 0, \quad (\dot{\mathfrak{h}}^*,\omega_i)=0, \quad (\delta_i,\delta_j)=0, \quad (\omega_i,\omega_j)=0, \quad (\delta_i,\omega_j)=\delta_{ij}.
\]
This form is nondegenerate and $\mathcal{W}$-invariant.

For $m=(m_1, \cdots ,m_n) \in \mathbb{Z}^n$, we define
\[
\delta_m = \sum m_i\delta_i, \quad (\delta_m,\delta_n)=0,
\]
where $\delta_m$ is called a \emph{null root}.

The set of \emph{real roots} is denoted by $\Delta^{\mathrm{re}}=\{\alpha+\delta_m \mid \alpha\in \dot{\Delta}, \; m \in \mathbb{Z}^n\}$ and the set of \emph{null roots} is denoted by $\Delta^{\mathrm{null}}= \{\delta_m \mid m\in \mathbb{Z}^n\}$. Then the root system is $\Delta= \Delta^{\mathrm{re}}\cup \Delta^{\mathrm{null}}$. For each real root $\gamma=\alpha+\delta_m$, we define its \emph{coroot} $\gamma^\vee=\alpha^\vee + \sum_{i=1}^{n}m_ic_i$. Then $(\gamma,\gamma)=2$ and $\gamma(\gamma^\vee)=2$.

Let $\dot{\mathfrak{g}} = \mathfrak{sl}_d(\mathbb{C})$ with decomposition $\dot{\mathfrak{g}} = \bigoplus_{\alpha\in\dot{\Delta}\cup\{0\}}\dot{\mathfrak{g}}_{\alpha}$.\\
We define $\tilde{\tau}_{\alpha+\delta_m} = \dot{\mathfrak{g}}_{\alpha}\otimes \mathbb{C}t^m$ and $$\tilde{\tau}_{\delta_m} = 
\begin{cases}
    \dot{\mathfrak{h}}\otimes\mathbb{C}t^m \oplus I\otimes\mathbb{C}t^m  &  m\in\mathbb{Z}^n\setminus \operatorname{rad} f,\\
    \mathfrak{h'}\otimes \mathbb{C}t^m \oplus HC_1(\mathbb{C}_q)_m  &  m \in \operatorname{rad} f \setminus \{0\}.\\
\end{cases}$$
Then
\[
\hat{\tau}(d,q) = \mathfrak{h}\oplus\bigoplus_{\alpha\in\Delta}\tilde{\tau}_\alpha
\]
is the corresponding root space decomposition of $\hat{\tau}(d,q)$.

For any real root $\gamma\in \Delta^{\mathrm{re}}$, we define the reflection on $\mathfrak{h}^*$ by $r_\gamma(\lambda) = \lambda-\lambda(\gamma^\vee)\gamma$, $\lambda\in\mathfrak{h^*}$.
The group generated by these reflections is the Weyl group $\mathcal{W}$. It can be easily verified that the form on $\mathfrak{h^*}$ is $\mathcal{W}$-invariant.

For $\lambda, \mu \in \dot{\mathfrak{h}}^*$ we write $\lambda \le \mu \Leftrightarrow \mu -\lambda \in \dot{Q}^+$, i.e. if and only if $\mu -\lambda$ is a non-negative integer combination of the simple roots.

\begin{definition}
    A module $V$ over $\hat{\tau}(d,q)$ is called integrable if 
    \begin{enumerate}
        \item ($V$ is weight module) $V = \bigoplus_{\lambda \in \mathfrak{h}^*} V_\lambda \; \text{where} \; V_\lambda= \{v \in V \mid h\cdot v=\lambda(h)v, \;\forall h \in \mathfrak{h}\}$,
        \item for every $\alpha\in\dot{\Delta},\; m\in \mathbb{Z}^n,\; v\in V, \; \exists\; k=k(\alpha,m,v)$ such that $(\dot{\mathfrak{g}}_\alpha\otimes t^m)^k\cdot v=0$.
    \end{enumerate}
\end{definition}

We denote by $P(V)$ the set of all weights for a weight module $V$. Then we have the following standard lemma.

\begin{lemma}\label{lemma 3.2} {\normalfont(\cite[Lemma 3.2]{[18]})} \hfill \break
    Let $V$ be an irreducible integrable module for $\hat{\tau}(d,q)$. Then:
    \begin{enumerate}
        \item $P(V)$ is $\mathcal{W}$-invariant;
        
        \item $\dim V_\lambda = \dim V_{\omega\lambda}$ for all $w \in \mathcal{W}$ and $\lambda \in P(V)$.
        
        \item $\lambda(\alpha^\vee)\in \mathbb{Z}$ for all $\alpha \in \Delta^{\mathrm{re}}, \; \lambda \in P(V)$.
        
        \item For $\alpha\in \Delta^{re},\; \lambda\in P(V)$ with $\lambda(\alpha^\vee)>0,$ we have $\lambda-\alpha\in P(V)$.
        
        \item For all $\lambda\in P(V),\; \lambda(c_i)$ is a constant integer.
        
    \end{enumerate}
\end{lemma}

\begin{lemma} \label{lemma 3.3}
    Let $V$ be an irreducible integrable module with finite-dimensional weight spaces for $\hat{\tau}(d,q)$. Then there exists $ \lambda \in P(V) \text{ such that } V_{\lambda+\eta}=0 \; \text{ for all } \eta \in \dot{Q}_+\setminus \{0\}$.  
\end{lemma}

\begin{proof}
    The proof is on the similar lines as in the proof of ({\cite[Proposition 2.4(ii)]{[4]}}), ({\cite[Proposition 2.12]{[12]}}) because we have the same Weyl group. Let $\lambda \in P(V)$ and $\underline{\lambda} = (\lambda(d_1), \cdots ,\lambda(d_n))$ and define $V_{\underline{\lambda}} = \{v \in V_\mu \mid \mu(d_i) = \lambda(d_i),\; 1 \leq i \leq n\}$. Since the $d_i$ act semisimply and commute, we have
\[
V = \bigoplus_{\underline{m}\in\mathbb{Z}^n} V_{\underline{\lambda}+\underline{m}},
\]
and each $V_{\underline{\lambda}+\underline{m}}$ is a $\dot{\mathfrak g}$-submodule. 
In particular, $V_{\underline{\lambda}}$ is an integrable $\dot{\mathfrak{g}}$-module with finite-dimensional weight spaces and hence decomposes into a direct sum of finite-dimensional irreducible $\dot{\mathfrak{g}}$-modules.
    
    Let $P(V_{\underline{\lambda}})=\{\mu \in P(V) \mid V_\mu \subseteq
     V_{\underline{\lambda}}\}$ then $P(V_{\underline{\lambda}}) \subset P(V)$.
     By \cref{lemma 3.2}, we may choose $\lambda_0 \in P(V_{\underline{\lambda}}) \; \text{such that } \; \lambda_0(\alpha_i^\vee) \in \mathbb{Z}^+, \; \forall i=1,\cdots ,n$.
     Since $\dim V_{\lambda_0}$ is finite, therefore the subspace $\mathcal{U}(\mathfrak{n}^+)V_{\lambda_0} $ is finite-dimensional. This implies that there exists an element $\eta \in \dot{Q}_+$ such that 
     \begin{equation} \label{equation *}
         \mathcal{U}(\mathfrak{n}^+)_\eta V_{\lambda_0} \neq 0, \quad \mathcal{U}(\mathfrak{n}^+)_{\eta'} V_{\lambda_0} = 0 \quad \forall \eta' > \eta.
     \end{equation}
     We say $V_{\lambda_0+\eta+\mu} =0 \; \text{for all} \; \mu \in \dot{Q}_+ \setminus \{0\}$, otherwise $V_{\lambda_0 + \eta + \mu} \neq 0 \; \text{for some} \; \mu \in \dot{Q}_+ \setminus \{0\}$. Choose $0\ne v\in V_{\lambda_0+\eta'+\mu}$ with 
$\mathfrak n^+\!\cdot v=0$, where $\eta'>\eta$.  
Then $M=\mathcal{U}(\dot{\mathfrak g})v$ is a finite-dimensional irreducible 
$\dot{\mathfrak g}$–module with highest weight 
$\lambda_0+\eta'+\mu$.  
By Lemma~2.7 of \cite{[4]}, $M\cap V_{\underline{\lambda}}\ne0$.  
Let $w\in M\cap V_{\underline{\lambda}}$ be a nonzero weight vector of weight $\lambda_0$. Since $M=\mathcal{U}(\dot{\mathfrak g})w$ and 
$\mathrm{wt}(v)-\mathrm{wt}(w)=\eta'+\mu\in\dot Q_+$, 
we have 
$v\in\mathcal U(\mathfrak n^+)_{\eta'+\mu}V_{\lambda_0}$, 
contradicting \eqref{equation *}. Hence $V_{\lambda_0 + \eta + \mu } =0 \; \text{for all} \; \mu \in \dot{Q}_+ \setminus\{0\}$.    
\end{proof}

We define $P^+(V)= \{\lambda \in P(V) \mid V_{\lambda+ \eta} =0, \; \forall \, \eta \in \dot{Q}_+ \setminus \{0\}\}$, then by \cref{lemma 3.3}, $P^+(V)$ is non-empty.

\section{action of central extension part}
In this section, we will show that if each central element $c_i \; (1 \leq i \leq n)$ acts trivially on an irreducible integrable $\hat{\tau}(d,q)$-module $V$ with finite-dimensional weight spaces, then the entire central extension $HC_1(\mathbb{C}_q)$ acts trivially on $V$.

Let $V$ be such a module. Throughout the paper, we set $\mathfrak{g} = \mathfrak{gl}_d(\mathbb{C})$ with the triangular decomposition
$\mathfrak{g} = N^-\oplus H \oplus N^+$,
where
$N^-$ is the subalgebra of strictly lower triangular matrices,
$H$ is the Cartan subalgebra consisting of diagonal matrices and
$N^+$ is the subalgebra of strictly upper triangular matrices.

\begin{proposition}\label{proposition 4.1}
    Suppose $c_i, \; 1 \leq i \leq n$ acts trivially on V. Then there exists non-zero $v,w \in V \; \text{such that} \; N^+ \otimes \mathbb{C}_q \cdot v =0, \; N^- \otimes \mathbb{C}_q \cdot w =0.$
\end{proposition}

\begin{proof}
    Let $\lambda \in P^+(V)$ then $V_{\lambda + \eta} =0 \; \text{for all} \; \eta \in \dot{Q}_+ \setminus\{0\}.$ If $V_{\lambda + \alpha + n\delta} =0 \; \text{for all} \; \alpha \in \dot{\Delta}_+$ and for all $n \in \mathbb{Z} $ then the Proposition follows.
    
    Let $V_{\lambda + \alpha + n\delta} \neq 0$ for some $\alpha \in \dot{\Delta}_+ \; and \; n \in \mathbb{Z}. \; \text{Set} \; \mu = \lambda + \alpha + n\delta$, then $V_{\mu + \beta + s\delta} =0 \; \text{for all} \; \beta \in \dot{\Delta}_+ \; \text{and all} \; s \in \mathbb{Z}$. Since $\alpha, \beta \in \dot{\Delta}_+$, therefore $(\alpha + \beta \mid \alpha) \text{ or } (\alpha + \beta \mid \beta)$ is positive. Suppose $(\alpha + \beta \mid \alpha) >0$. Set $\gamma = \alpha + (s + n)\delta$, then $(\mu + \beta + s\delta \mid \gamma) >0$ and hence by {\cite[Lemma 2.2]{[4]}} $V_{\mu + \beta + s\delta - \gamma} = V_{\lambda + \beta} \neq 0$ contradicting the fact that $\lambda \in P^+(V)$, therefore $N^+ \otimes \mathbb{C}_q\cdot V_\mu =0.$\\
    Similarly we can prove $N^- \otimes \mathbb{C}_q\cdot w =0$ for some $w \in V$.
\end{proof}

\begin{proposition}\label{proposition 4.2}
    Let V be an irreducible integrable module for $\hat{\tau}(d,q)$ with finite-dimensional weight spaces. Let k be as in \cref{proposition 2.4} above. Suppose $k \geq 1 \; \text{and} \; c_i=0 \; \forall \; i=1, \cdots,n$; then such a module V does not exist.
\end{proposition}

\begin{proof}
    The proof of this Proposition will be parallel to (Proposition 4.13,\cite{[3]}). Here we will construct different Heisenberg Lie algebra.\\
    Let $\langle t_i,t^rt_i^{-1} \rangle \neq 0$ then by \cref{lemma 2.2}, $r \in \operatorname{rad}f$ and by \cref{proposition 2.4}, we have for $i \geq k+1 \text{ and } r_{k+1}= \cdots = r_n =0.$ Since $\mathbb{C}_q$ is rational, choose $N_i$ to be the minimum integer such that $t_i^{N_i} \in Z(\mathbb{C}_q).$ Now consider, $H = \text{Span}_{\mathbb{C}} \{h \otimes t^rt_i^{N_ik}, h \otimes t_i^{-N_ik}, \langle t_i,t^rt_i^{-1} \rangle \mid k >0\}$ the Heisenberg algebra and we construct an infinite set of linearly independent vectors inside a finite weight space $\{Yt_i^{-N_ir} \; Yt_i^{N_ir}w, \; r>0, \; r \neq \pm l_0\}$ which contradicts the fact that each weight space is of finite dimension.
\end{proof}

\begin{lemma}
    $HC_1(\mathbb{C}_q)$ acts trivially on V.
\end{lemma}

\begin{proof}
    We recall $HC_1(\mathbb{C}_q) = \{ \langle t^a,t^b \rangle \mid a + b \in \operatorname{rad}f\}$. By \cref{proposition 4.2}, if $V$ is an irreducible integrable module with finite-dimensional weight spaces and $c_i = 0 \; \text{for all} \; i = 1, \cdots ,n$ then the number of central operators acting non-trivially on $V$ is $k = 0$. We note that $\langle t^a,t^b \rangle$ is a central operator of degree $a+b$ if $a + b \in \operatorname{rad}f$,
    hence our lemma.
\end{proof}

\section{highest weight space}
In this section, we introduce the notion of \textit{highest central operators} on the highest weight space $V^+$. We then use these operators to show that $V$ admits a finite-dimensional irreducible quotient as a module over $\mathfrak{g} \otimes \mathbb{C}_q$.
We recall from (1.10, \cite{[18]}) that $\mathfrak{gl}_d(\mathbb{C}) \otimes \mathbb{C}_q = \mathfrak{sl}_d(\mathbb{C}_q) \oplus (I \otimes Z(\mathbb{C}_q))$ and the kernel is central. We can treat any $\mathfrak{sl}_d(\mathbb{C}_q)$ module as a $(\mathfrak{gl}_d(\mathbb{C}) \otimes \mathbb{C}_q)$ module where the additional center $I \otimes Z(\mathbb{C}_q)$ acts trivially.

Let $V$  be an irreducible integrable graded module for $\hat{\tau}(d,q)$ with finite-dimensional weight spaces and assume that $HC_1(\mathbb{C}_q)$ acts trivially on $V$. In general such a module is referred to as a non-graded module for $\tau(d,q)$.

We will establish a one-to-one correspondence between graded and non-graded modules. Starting with a graded integrable module with finite-dimensional weight spaces, we obtain a finite-dimensional non-graded module. It is easy to classify irreducible finite-dimensional non-graded modules and this reduction forms the basis of the classification theory developed in the subsequent sections.

For this purpose, we define $V^+ = \{v \in V \mid N^+ \otimes \mathbb{C}_q \cdot v =0\}$, which is clearly nonzero by \cref{proposition 4.1}. We shall refer to this space as the highest weight space of $V$.

\begin{proposition}\label{proposition 5.2}
\leavevmode
\begin{enumerate}
    \item $V^+$ is an irreducible $(H \otimes \mathbb{C}_q) \oplus D$-module.
    \item $V=\mathcal{U}(N^-\otimes\mathbb{C}_q)V^+$.
\end{enumerate}
\end{proposition}

\begin{proof}
    We note that $[N^+ \otimes \mathbb{C}_q, H \otimes \mathbb{C}_q] \subseteq N^+ \otimes \mathbb{C}_q$ and $[d_i,X \otimes t^a] = a_i X \otimes t^a $ and with the help of these, we can prove that $(H \otimes \mathbb{C}_q \oplus D)\cdot V^+ \subseteq V^+$. So $V^+$ is a $H \otimes \mathbb{C}_q \oplus D$-module and by PBW theorem, $\mathcal{U}(\mathfrak{gl}_d(\mathbb{C}) \otimes \mathbb{C}_q) = \mathcal{U}(N^- \otimes \mathbb{C}_q) \mathcal{U}(H \otimes \mathbb{C}_q) \mathcal{U}(N^+ \otimes \mathbb{C}_q)$. 
    Since $V$ is an irreducible module, we have $V = \mathcal{U}(N^- \otimes \mathbb{C}_q)V^+$.
    It remains to prove that $V^+$ is irreducible.
    Suppose $W \subseteq V^+$ is a nonzero $H \otimes \mathbb{C}_q \oplus D$-submodule. Then $\mathcal{U}(N^- \otimes \mathbb{C}_q) W$ is a $(\mathfrak{gl}_d(\mathbb{C})\otimes \mathbb{C}_q)\oplus D$-submodule of $V$. Then by the irreducibility of $V$, it follows that $V^+ = W$. Hence $V^+$ is irreducible.  
\end{proof}

\begin{lemma}\label{lemma 5.3}
\leavevmode
    \begin{enumerate}
        \item There exist unique $\lambda \in H^*$ and $\beta \in \mathbb{C}^n$(not necessarily unique) such that the weights of $V^+$ are of the form $\lambda + \delta_{r + \beta}$, where $r \in \mathbb{Z}^n$.

        \item $\lambda$ is dominant integral.
    \end{enumerate}
\end{lemma}

\begin{proof}
    \begin{enumerate}
        \item Since $H \otimes 1$ commutes with $(H \otimes \mathbb{C}_q) \oplus D$, it follows from \cref{proposition 5.2} that $H \otimes 1$ acts on $V^+$ by scalars. Equivalently, the action is determined by a single linear functional on $V^+$, which we denote by $\lambda$. Concretely, for each $h \in H$ there exists a scalar $\lambda(h)$ such that
        \[
        (h \otimes 1)\cdot v = \lambda(h)\cdot v \qquad \forall \; v \in V^+.
        \]
        For the derivations, recall that
        \[
        [d_i, x \otimes t^r] = r_i x \otimes t^r.
        \]
        Thus if $v_0 \in V^+$ has a $d$-eigenvalue $\beta=(\beta_1, \cdots ,\beta_n) \in \mathbb{C}^n$, i.e. $d_i v_0 = \beta_i v_0$, then every weight vector in $V^+$ has $d$-eigenvalues of the form $\beta + r$ for some $r \in \mathbb{Z}^n$. Consequently
        \[
        P(V^+) \subseteq \{ \lambda + \delta_{\beta + r} \mid r \in \mathbb{Z}^n \}.
        \]

        \item Let $v_0 \in V^+$ has weight $\mu = \lambda + \delta_{\beta+r}$ for some $r \in \mathbb{Z}^n$ (so $\mu \in P(V) \; \text{and} \; \mu|_H = \lambda$). If $\lambda(\alpha^\vee) <0$ for some $\alpha \in \Delta^{\text{re}}$ then by \cref{lemma 3.2}, $\lambda(\alpha^\vee) \in \mathbb{Z}$ and $\lambda + \alpha \in P(V)$, which is a contradiction because $v_0$ is a highest weight vector. Hence $\lambda(\alpha^\vee) \in \mathbb{Z}_{\ge 0}$ for every real root $\alpha$, that is $\lambda$ is dominant integral weight.
    \end{enumerate}
\end{proof}

\begin{remark}\label{remark 5.4}
\leavevmode
    \begin{enumerate}
        \item Suppose $\lambda = 0$ in \cref{lemma 5.3}, then $V^+$ contains a highest weight vector of weight $\delta_\beta$ for some $\beta \in \mathbb{C}^n$. Since $V$ is irreducible and integrable, each $\mathfrak{sl}_2$-subalgebra corresponding to a real root $\alpha$ acts trivially $($as $\lambda(\alpha^\vee)=0$, by \cref{lemma 3.2}$)$, so all weights coincide with $\delta_\beta$. Using the commutator relations and the fact that $HC_1(\mathbb{C}_q)$ acts trivially, it follows that the Cartan elements $h_\alpha\otimes t^m$ also act by $0$; hence the entire subalgebra $\mathfrak{g} \otimes \mathbb{C}_q$ acts trivially on $V$. Consequently $V$ is one-dimensional and isomorphic to the highest weight module $V(\delta_\beta)$ with the highest weight $\delta_\beta$.
        
        \item For notational simplicity, we may take, without loss of generality, the vector $\beta$ in \cref{lemma 5.3} to be the zero vector.
    \end{enumerate}
\end{remark}

To proceed further, for each $j=1,\cdots ,n$, consider the subalgebra of $\mathfrak{g}\otimes Z(\mathbb{C}_q) \oplus HC_1(\mathbb{C}_q)\oplus D$ defined by
\[
L_j = \mathfrak{g}\otimes \mathbb{C}[t_j^{N_j},t_j^{-N_j}] \oplus \mathbb{C}c_j\oplus \mathbb{C}d_j, \quad \text{where} \; c_j=\langle t_j,t_j^{-1}\rangle.
\]
Let $\mathcal{W}_j$ and $\mathcal{W}_0$ be the Weyl groups related to the affine Kac-Moody algebra $L_j$ and the simple finite-dimensional Lie algebra $\mathfrak{g}$ respectively. Let $\theta$ denote the highest root of $\mathfrak{g}$ and let $\theta^\vee$ be the corresponding co-root. There is a natural isomorphism $\gamma : H \to H^*$ induced by the nondegenerate invariant bilinear form $(\; | \;)$ on $\mathfrak{g}$, defined by
\[
\gamma(h)(x)=(x|h), \quad \forall \, x \in H.
\]
For each $j$, we denote by $\gamma_j$ this same identification on $H$; i.e. $\gamma_j=\gamma$.
Using this identification, set
\[
M_j=\gamma_j(\mathbb{Z}(\mathcal{W}_0\theta^\vee)), \quad T_j=\{t_{\alpha_j} \mid \alpha_j \in M_j\}. 
\]
Then by \cite[Proposition 6.5]{[8]}, there is an isomorphism of groups
\[
\mathcal{W}_j \cong \mathcal{W}_0 \ltimes T_j.
\]

\begin{remark}\label{remark 5.2}
    If we set $\widehat{H}_j=H \oplus \mathbb{C}d_j$, then as $c_j$ acts trivially on $V$, it follows from \cite[Eqn 6.5.5]{[6]} that the translation $t_{\alpha_j}$ acts on $\widehat{H}_j^*$ as
    \[
    t_{\alpha_j}(\mu_j)=\mu_j-\mu_j(\alpha_j^\vee)\delta_j, \quad \forall \; \mu_j \in \widehat{H}_j^*, \; \alpha_j\in M_j.
    \]
\end{remark}

\begin{proposition}\label{proposition 5.5}
    There are only finitely many $H \otimes Z(\mathbb{C}_q)$-submodules of $V^+$.
\end{proposition}

\begin{proof}
    For each $j=1, \cdots,n$, we have
    \[
    \frac{2\theta}{(\theta | \theta)}= \gamma_j(\theta^\vee) \in M_j.
    \]
    Set $\theta_j=\gamma_j(\theta^\vee)$ and $p_j=(\lambda | \theta_j)\in \mathbb{N}$ (by \cref{lemma 5.3} and \cref{remark 5.4}).

    For any $s=(s_1, \cdots ,s_n) \in \mathbb{Z}^n$, define
    \[
    s(m) = (s_1m_1, \cdots ,s_nm_n)\in \operatorname{rad}f.
    \]
    Again for any integer $k_j \in \mathbb{Z}$ with $|k_j|\ge m_jp_j$, there exist non-zero integers $q_j \in \mathbb{Z}$ and $r_j \in \mathbb{Z}_+$ such that
    \[
    k_j=q_jp_jm_j+r_j, \; |r_j|<|p_jm_j|.
    \]
    Let $t_j=t_{\theta_j}$ and denote by $\mathcal{W}'$, the Weyl group corresponding to $\mathfrak{g} \otimes Z(\mathbb{C}_q) \oplus HC_1(\mathbb{C}_q) \oplus D$. Define
    \[
    w=(\prod_{q_j<0}t_j^{-m_jq_j})(\prod_{q_j>0}t_j^{m_jq_j}) \in \mathcal{W}',
    \]
    then by \cref{remark 5.2}, for $k=(k_1, \cdots ,k_n)$ and $r=(r_1, \cdots ,r_n)$ we have
    \[
    w(\lambda+\delta_{k})=\lambda+\delta_{r}, \; |r_i|<|p_im_i|, \quad \forall \; i=1, \cdots ,n.
    \]
    Now set
    \[
    P^{(1)}=\{\lambda+\delta_{r} \mid r\in \mathbb{Z}^n, \; |r_i|<|p_im_i|\}.
    \]
    Finally, observe that any \(H\otimes Z(\mathbb{C}_q)\)-submodule of \(V^+\) is determined by which highest-weight vectors of \(V\) it contains, therefore it is generated by those highest-weight vectors. By the above Weyl translation argument and the finite dimensionality of the weight spaces, each submodule is generated by a subset of the basis vectors with weights in \(P^{(1)}\). Since \(P^{(1)}\) is finite, there are only finitely many such submodules. Hence the Proposition follows.
\end{proof}

\begin{remark}
    By the argument above, every weight $\mu \in P(V^+)$ is $\mathcal{W}'$-conjugate to some $\nu \in P^{(1)}$. Since Weyl translations preserve dimensions and each $\dim V_\nu$ is finite, we have
    \[
    \dim V_\mu^+ \le \max_{\nu \in P^{(1)}}\dim V_\nu.
    \]
    In particular, $V^+$ has a uniform finite bound on the dimensions of all its weight spaces.
\end{remark}

\begin{lemma}\label{lemma 5.8}
    There exist only finitely many $\lambda_1, \cdots ,\lambda_l \in H^*$ and some $\beta \in \mathbb{C}^n$ such that $P(V) \subseteq \{\lambda_i+\delta_{\beta+r} \mid 1 \le i \le l, \; r \in \mathbb{Z}^n\}$.
\end{lemma}

\begin{proof}
    By \cref{remark 5.4} and \cref{proposition 5.2}, the set of weights of $V$ satisfies
    \[
    P(V) \subseteq \{\lambda - \eta + \delta_r \mid \eta \in \dot{Q}^+, \; r \in \mathbb{Z}^n\},
    \]
    where $\lambda$ is the unique element of $H^*$ as in \cref{lemma 5.3}. Fix a weight of the form $\mu = \lambda -\eta + \delta_r \in P(V)$ (with $\eta \in \dot{Q}^+, \; r \in \mathbb{Z}^n$). By Lemma A of \cite[Section 13.2]{[7]}, there exists a unique dominant integral weight $\mu^+ \in H^*$ and some $\sigma \in \mathcal{W}_0$ such that
    \[
    \sigma(\lambda-\eta)= \mu^+.
    \]
    Since $\mathcal{W}_0 \subseteq \mathcal{W}$, it follows that $\sigma \in \mathcal{W}$ such that
    \[
    \sigma(\lambda-\eta+\delta_r) = \mu^+ + \delta_r.
    \]
    Moreover by \cref{lemma 3.2}, we have $\mu^+ \le \lambda$. Thus every weight $\mu$ of $V$ is $\mathcal{W}_0$-conjugate to a unique dominant weight $\mu^+$ with $\mu^+ \le \lambda$. By Lemma B of \cite[Section 13.2]{[7]}, only finitely many dominant integral weights $\mu^+$ satisfy this inequality. We denote them by $\{\lambda_1, \cdots ,\lambda_l\}$. Consequently
    \[
    P(V) \subseteq \{\lambda_i + \delta_{\beta+r} \mid 1 \le i \le l, \; r \in \mathbb{Z}^n\}.
    \]
\end{proof}

\begin{remark}\label{remark 5.8}
\leavevmode
    For each $i=1, \cdots ,l$, let
    \[
    P_i^{(1)}=\{\lambda_i+\delta_{\beta+r}\; ; \; |r_j|<|p_jm_j|,\; 1\le j\le n\},
    \]
    be the finite fundamental set constructed via the Weyl-group argument in \cref{proposition 5.5}. Every weight in the corresponding coset $\lambda_i+\delta_{\beta+\mathbb Z^n}$ lies in the $\mathcal{W}'$-orbit of an element of $P_i^{(1)}$. Since $V$ is integrable with finite-dimensional weight spaces, the $\mathcal{W}'$-action preserves weight multiplicities. Therefore, each weight in the coset has the same dimension as some weight in $P_i^{(1)}$. Setting
    \[
    C_i:=\max_{\nu\in P_i}\dim V_\nu<\infty \quad
    \text{and} \quad C=\max_{1\le i\le\ell}C_i,
    \] we obtain a uniform bound on the dimensions of all weight spaces
    \[
    \dim V_\mu \le C, \quad \text{for all weights} \; \mu \; \text{in}\; V.
    \]

\end{remark}

\begin{definition}
    A linear map $z : V^+ \to V^+$ is called a highest central operator of degree $m$ if it satisfies the following conditions.
    \begin{enumerate}
        \item $z$ commutes with the action of $H \otimes \mathbb{C}_q$,
        \item $d_iz-zd_i = m_iz \; \forall \; i=1,2,\cdots ,n$.
    \end{enumerate}
\end{definition}

\begin{remark}\label{remark 5.10}
\leavevmode
    \begin{enumerate}
        \item $H \otimes Z(\mathbb{C}_q)$ is an abelian Lie algebra.
        \item For each $h \in H$ and $r \in Z(\mathbb{C}_q)$, $h \otimes t^r$ is a highest central operator of degree $r$ on $V^+$.
    \end{enumerate}
\end{remark}

\begin{lemma}\label{lemma 5.1}
Let \(V\) be an integrable module for the loop algebra \(L = \mathfrak{g} \otimes \mathbb{C}[t, t^{-1}]\) (with no central extension).
If \(v \in V\) is a highest weight vector of weight \(\lambda\), then
\[
h \cdot v = \lambda(h)v = 0 \; \text{if and only if} \; (h \otimes t^k)\cdot v = 0 \; \text{for all} \; k>0.
\]
\end{lemma}

\begin{proof}
    We fix a positive (simple) root $\beta$ and write $h_\beta$ for the coroot. We define the formal power series as in \cite{[6]}
    \[
    \tilde{H}_{\beta}(u)=\sum_{m\in \mathbb{Z}}h_{\beta,m}u^{m+1},
    \]
    \[
    \Lambda_{\beta}^+(u) = \exp{\left(-\sum_{k\geq 1}\frac{h_\beta \otimes t^k}{k}u^k\right)} = \sum_{m\geq 0}\Lambda_{\beta,m}u^m,
    \]
    so the constant term $\Lambda_{\beta,0}=1$. By Proposition 1.1(i) of \cite{[6]}, for the highest weight vector $v$, we have $\Lambda_{\beta,m}\cdot v=0$ for all $m > \lambda(h_{\beta})$. If $\lambda(h_{\beta})=0$, then $\Lambda_{\beta,m}\cdot v=0$ for all $m>0$, hence
    \[
    \Lambda_\beta^+(u) \cdot v = \Lambda_{\beta,0}\cdot v = 1\cdot v =v.
    \]
    Since the constant term of $\Lambda_{\beta}^+(u)$ is $1$, its logarithm is well defined:
    \[
    \log{\Lambda_{\beta}^+(u)} = -\sum_{k \geq 1} \frac{h_\beta \otimes t^k}{k}u^k.
    \]
    Acting on $v$, we have $\Lambda_{\beta}^+(u)\cdot v =v$, hence $\log{\Lambda_{\beta}^+(u)}\cdot v =0$. Comparing coefficients of $u^k$ yields $(h_{\beta}\otimes t^k) \cdot v =0$ for all $k \geq 1$. Thus $\lambda(h_\beta)=0 \Rightarrow (h_\beta \otimes t^k) \cdot v =0$ for all $k>0$.

    Conversely assume $(h_\beta \otimes t^k) \cdot v =0$ for all $k>0$. Then the exponential in the definition of $\Lambda_\beta^+(u)$ acts trivially on $v$ so $\Lambda_\beta^+(u)\cdot v =v$, i.e. $\Lambda_{\beta,0}\cdot v=v \; \text{and}\; \Lambda_{\beta,m} \cdot v =0$ for all $m\geq 1$. By Proposition 1.1(iii) of \cite{[6]}, we have
    \[
    \left(\tilde{H}_{\beta}(u)\Lambda_{\beta}^+(u)\right)_s\cdot v=0 \quad \text{for all} \; s \in \mathbb{Z},
    \]
    which can be written as
    \[
    \sum_{r=0}^{\infty}\left(\tilde{H}_{\beta}(u)\right)_{s-r}\Lambda_{\beta,r} \cdot v = \sum_{r=0}^{\infty}h_{\beta,s-r-1}\Lambda_{\beta,r}\cdot v = h_{\beta,s-1}\Lambda_{\beta,0}\cdot v = 0.
    \]
    Thus $h_{\beta,s-1}\cdot v = 0$ for all $s \in \mathbb{Z}$. In particular for $s=1$, i.e. $h_{\beta}\cdot v = \lambda(h_{\beta})v =0$. This completes the proof.
\end{proof}

\begin{lemma}\label{lemma 5.12}
Let $h \in \dot{\mathfrak{h}}$ and $a \in \operatorname{rad}f$.  
If there exists $v \in V^+$ such that $(h \otimes t^a)\cdot v \neq 0$, then 
$
(h \otimes t^a)\cdot w \neq 0 \quad \text{for all} \; w \in V^+.
$
\end{lemma}

\begin{proof}
  We consider $W' = \{v \in V^+ \mid h \otimes t^a\cdot v =0,\; h \in \dot{\mathfrak{h}}, \; a \in \operatorname{rad}f\}$ (for fixed h). Since $h \otimes t^a$ commutes with $H \otimes \mathbb{C}_q$ and $[d_i,h \otimes t^a] = a_i h \otimes t^a$, therefore $W'$ is a $(H \otimes \mathbb{C}_q) \oplus D$-module. Since $V^+$ is irreducible, $W'$ must be either $0$ or $V^+$.
\end{proof}

\begin{lemma}\label{lemma 5.13}
For each $1 \le i\le n$, there exists a nonzero bijective highest central operator
\[
z_i=h \otimes t_i^{M_i}, \quad h\in H, \; M_i\in \mathbb{Z},
\]
acting on $V^+$.
\end{lemma}

\begin{proof}
Fix $i \in \{1, \cdots ,n\}$ and choose $0 \ne v_0 \in V^+$. 
By \cref{lemma 5.1}, working in the loop algebra 
$\mathfrak{gl}_d(\mathbb{C}) \otimes \mathbb{C}[t_i^{N_i}, t_i^{-N_i}]$, there exists $k \in \mathbb{Z}_+$ and an element $h \in H$ with $\lambda(h)\ne 0$ such that
\[
(h \otimes t_i^{k N_i}) \cdot v \neq 0.
\] 
Set $M_i = k N_i$ and define $z_i = h \otimes t_i^{M_i}$, then $z_i$ acts nontrivially on $v_0$.\\
By \cref{lemma 5.12}, if a central element $h \otimes t^a$ acts nontrivially on some nonzero vector of $V^+$, it acts nontrivially on every nonzero vector. Hence $z_i$ is injective on $V^+$. 
Since $z_i V^+$ is a nonzero $H \otimes \mathbb{C}_q$-submodule of the irreducible module $V^+$, we must have $z_i V^+ = V^+$, therefore $z_i$ is bijective.
\end{proof}

\begin{lemma}
Let $z_i : V^+ \to V^+$ be a highest central operator of degree $M_i e_i$. Define a linear operator $T_i : V^+ \to V^+$ by 
\[
T_i(z_i v) = v \quad \forall \; v \in V^+.
\] 
Further define
\[
W = \operatorname{Span}_{\mathbb{C}} \{ z_i v - v \mid v \in V^+, \; 1 \leq i \leq n \}.
\]
Then the following statements hold
\begin{enumerate}
    \item $T_i z_i = \mathrm{Id}$.
    \item $[T_i, H \otimes \mathbb{C}_q] = 0$.
    \item For all $v \in V^+$, $T_i v - v \in W$.
    \item For all $k \ge 1$ and $v \in V^+$, $T_i^k v - v \in W$.
    \item For all $k \ge 1$ and $v \in V^+$, $z_i^k v - v \in W$.
\end{enumerate}
\end{lemma}

\begin{proof}
Statements (1) and (2) are immediate from the definitions.  

(3) For any $v \in V^+$, there exists $w \in V^+$ such that $z_i w = v$. Then
\[
T_i v - v = T_i (z_i w) - z_i w = w - z_i w = z_i(-w)-(-w)\in W,
\]
hence $T_i v - v \in W$.

(4) We proceed by induction on $k$. The base case $k = 1$ is exactly (3). Assume the statement holds for $k-1$, then for $v \in V^+$,
\begin{align*}
T_i^k v - v &= T_i(T_i^{k-1} v) - v \\
&= \big(T_i(T_i^{k-1} v) - T_i^{k-1} v\big) + \big(T_i^{k-1} v - v\big) \in W,
\end{align*}
because both terms belong to $W$ by the induction hypothesis and (3).  

(5) The proof is similar to (4) and follows by induction on $k$.
\end{proof}

\begin{proposition}\label{proposition 5.12}
    $W = \text{Span}_{\mathbb{C}}\{z_iv-v \mid  v \in V^+, \; 1 \leq i \leq n\}$ is a proper $H \otimes \mathbb{C}_q$-submodule of $V^+$. 
\end{proposition}

\begin{proof}
For each \(k = 1, \dots, n\), define
\[
D_k = \operatorname{Span}_{\mathbb{C}} \{ d_{k+1}, \dots, d_n \} \quad \text{and} \quad
W_1 = \operatorname{Span}_{\mathbb{C}} \{ z_1 v - v \mid v \in V^+ \}.
\]  
We first show that \(W_1\) is a proper \((H \otimes \mathbb{C}_q) \oplus D_1\)-submodule of \(V^+\).

Suppose not, then there exists some nonzero weight vector \(v_0 \in V^+\) of weight \(\mu\) lies in \(W_1\). So there exist finitely many nonzero weight vectors \(v_1, \dots, v_m \in V^+\) (which may be assumed to lie in distinct weight spaces) and nonzero scalars \(c_1, \dots, c_m \in \mathbb{C}\) such that
\[
v_0 = \sum_{i=1}^m c_i (z_1 v_i - v_i).
\]

Using elementary weight considerations, we may assume (without loss of generality) that 
\[
c_i (z_1 v_i) = c_{i+1} v_{i+1}, \quad \text{for } i = 1, \dots, m-1.
\]  
Then 
\[
v_0 = c_m (z_1 v_m) - c_1 v_1.
\]  
Since \(z_1\) has nonzero degree, \(z_1 v_m\) and \(v_1\) lie in distinct derivation weight spaces. Because the decomposition into different weight spaces is direct, a nonzero vector from one weight space cannot be equal to a nonzero combination of vectors from another weight space. Hence we must have \(c_1 = 0\) or \(c_m = 0\), a contradiction. This shows that \(W_1\) is proper.

Now consider 
\[
W_2 = \operatorname{Span}_{\mathbb{C}} \{ z_2 v - v \mid v \in V^+/W_1 \}.
\]  
Repeating the same argument, we can prove \(W_2\) is a proper \((H \otimes \mathbb{C}_q) \oplus D_2\)-submodule of \(V^+/W_1\). Iterating this process for \(k = 1, \dots, n\), we conclude that
\[
W = \operatorname{Span}_{\mathbb{C}} \{ z_i v - v \mid v \in V^+, \; 1 \le i \le n \}
\]  
is a proper \(H \otimes \mathbb{C}_q\)-submodule of \(V^+\), as claimed.
\end{proof}

\begin{theorem}\label{theorem 5.16}
    Let $V' = \mathcal{U}(\mathfrak{g}\otimes \mathbb{C}_q)W$ and $\overline{V}' = V/V'$. Then
    \begin{enumerate}
        \item $V'$ is a proper $\mathfrak{g} \otimes \mathbb{C}_q$-submodule of $V$.

        \item $\overline{V}'$ is a finite-dimensional $\mathfrak{g}\otimes \mathbb{C}_q$-module.

        \item $V$ admits a finite-dimensional irreducible quotient over $\mathfrak{g}\otimes \mathbb{C}_q$, i.e. there exists a $\mathfrak{g}\otimes \mathbb{C}_q$-submodule $V''$ of $V$ containing $V'$ such that $V/V''$ is a finite-dimensional irreducible module for $\mathfrak{g}\otimes \mathbb{C}_q$. 
    \end{enumerate}
\end{theorem}

\begin{proof}
    \begin{enumerate}
        \item This follows from \cref{proposition 5.2} and \cref{proposition 5.12}.

        \item By \cref{lemma 5.8}, the set of weights $P(V)$ is contained in a finite union of $\delta$–cosets; hence there exist finitely many representatives $\mu^{(1)}, \dots, \mu^{(s)}$ such that
\[
V = \bigoplus_{j=1}^s \widetilde{V}_{\mu^{(j)}}, \quad \text{where} \quad
\widetilde{V}_{\mu^{(j)}} = \bigoplus_{r \in \mathbb{Z}^n} V_{\mu^{(j)} + \delta_r}.
\]
Since $V'$ is also $\delta$–graded, we obtain corresponding decompositions
\[
V' = \bigoplus_{j=1}^s \widetilde{V}'_{\mu^{(j)}}, \qquad \text{where} \quad
\widetilde{V}'_{\mu^{(j)}} = \widetilde{V}_{\mu^{(j)}} \cap V'.
\]
Consequently
\[
\overline{V} = V / V' \cong \bigoplus_{j=1}^s \big( \widetilde{V}_{\mu^{(j)}} / \widetilde{V}'_{\mu^{(j)}} \big),
\]
and it suffices to prove that for a fixed $\mu$, the quotient $\widetilde{V}_\mu / \widetilde{V}'_\mu$ is finite-dimensional.

Fix such a $\mu$ and let $v \in \widetilde{V}_\mu$, since $V$ is generated by the highest-weight space $V^+$, we may write
\[
v = \sum_{i=1}^m u_i w_i, \quad 
u_i \in \mathcal{U}(\mathfrak{g} \otimes \mathbb{C}_q), \; w_i \in V^+.
\]
Let $\mathcal{Z}$ denote the commutative subalgebra of $\mathcal{U}(H \otimes Z(\mathbb{C}_q))$ generated by the bijective highest-central operators $z_1, \dots, z_n$ introduced in \cref{lemma 5.13}. Each $z_j$ has homogeneous $\delta$–degree $p^{(j)} \in \mathbb{Z}^n$ and acts bijectively on $V^+$. 
Consider the collection of all vectors obtained by applying monomials in the $z_j$ to $w_i \in V^+$,
\[
\{ z_1^{k_1} \cdots z_n^{k_n} w_i \mid (k_1,\dots,k_n) \in \mathbb{Z}^n \}.
\]
Each such vector has weight $\operatorname{wt}(w_i)+\delta_{k_1 p^{(1)} + \cdots + k_n p^{(n)}}$. 
Since $t_j^{N_j} \in Z(\mathbb{C}_q)$ for every $j$, the degrees $p^{(j)}$ are periodic modulo $N_j$ in each coordinate and hence only finitely many distinct weights of this form occur.
Since each weight space of $V$ is finite-dimensional therefore
\[
\operatorname{span}\{\, z_1^{k_1} \cdots z_n^{k_n} w_i \mid (k_1, \dots, k_n) \in \mathbb{Z}^n \,\}
\subseteq \bigoplus_{|k_j| < N_j} V_{\mathrm{wt}(w_i) + \delta_{k_1 p^{(1)} + \cdots + k_n p^{(n)}}}
\]
is finite-dimensional. On this finite-dimensional space, each $z_j$ satisfies a nontrivial polynomial relation. Hence there exists $z'_i \in \mathcal{U}(\mathcal{Z})$ satisfying $(1 - z'_i) w_i \in W$.
Now we can rewrite each term in the expansion of $v$ as
\[
u_i w_i = u_i \big( (1 - z'_i) w_i \big) + u_i (z'_i w_i).
\]
By construction, $u_i \big( (1 - z'_i) w_i \big) \in \mathcal{U}(\mathfrak{g} \otimes \mathbb{C}_q) W = V'$, 
so these terms vanish in the quotient $\widetilde{V}_\mu / \widetilde{V}'_\mu$. 
Each $z'_i$ is a finite $\mathbb{C}$–linear combination of monomials in the commuting central operators $z_1, \dots, z_n$, say
\[
z'_i = \sum_{\alpha \in F_i} c_{\alpha} z_1^{\alpha_1} \cdots z_n^{\alpha_n},
\]
where $F_i \subset \mathbb{Z}^n$ is a finite index set and $c_{\alpha} \in \mathbb{C}$. 
Since each $z_j$ has degree $p^{(j)} \in \mathbb{Z}^n$, the degree of $z'_i$ is contained in the finite set 
\[
\mathrm{deg}(z'_i) = \{\, \alpha_1 p^{(1)} + \cdots + \alpha_n p^{(n)} 
   \mid \alpha \in F_i \,\} \subset \mathbb{Z}^n.
\]
If $u_i$ has homogeneous $\delta$–degree $r_i$, then every term of the form 
$u_i(z_1^{\alpha_1} \cdots z_n^{\alpha_n} w_i)$ lies in the weight space
\[
V_{\mu + \delta_{\,r_i + \alpha_1 p^{(1)} + \cdots + \alpha_n p^{(n)}}}.
\]
Hence all vectors $u_i(z'_i w_i)$ belong to weight spaces whose $\delta$–shifts from $\mu$ 
lie in the finite set 
\[
S_\mu = \bigcup_{i=1}^m \{\, r_i + \beta \mid \beta \in \mathrm{deg}(z'_i) \,\} 
   \subset \mathbb{Z}^n.
\]
It follows that the sum $\sum_{i=1}^{m}u_i(z'_i w_i)$ lies in the finite direct sum
\[
N(\mu) = \bigoplus_{\rho \in S_\mu} V_{\mu + \delta_\rho} \subset \widetilde{V}_\mu.
\]
Therefore $v$ is congruent modulo $\widetilde{V}'_\mu$ to an element of $N(\mu)$, 
which is a fixed finite-dimensional subspace of $\widetilde{V}_\mu$ depending only on $\mu$ 
and not on the choice of $v$. 
Since each weight space $V_{\mu + \delta_\rho}$ is finite-dimensional, 
so is $N(\mu)$.
Hence $\widetilde{V}_\mu / \widetilde{V}'_\mu$ is finite-dimensional for each $\mu$ 
and consequently $\overline{V} = V / V'$ is finite-dimensional. This completes the proof.
 
        \item We consider any strictly increasing sequence of $\mathfrak{g} \otimes \mathbb{C}_q$-submodules of $V$ containing $V'$, say
        $$
        V' \subset V_1 \subset \cdots .
        $$
        This yields a corresponding decreasing sequence of quotient modules over $\mathfrak{g} \otimes \mathbb{C}_q$, for which we have
        $$
        0 \neq \dim(V/V_1) < \dim(V/V').
        $$
        By (2), such a chain of submodules cannot continue indefinitely and therefore must terminate. Consequently there exists a proper maximal submodule $N \subset V$ containing $V'$. From this point on, we denote the irreducible quotient module $V/N$ by $\overline{V}$.

     \end{enumerate}
\end{proof}

\section{Recovering the original module}
In this section, our goal is to recover the irreducible $\hat{\tau}(d,q)$-module $V$ from the finite-dimensional irreducible $\mathfrak{g} \otimes \mathbb{C}_q$-module $\overline{V}$ defined at the end of the previous section.

Since $V$ is an irreducible integrable $\hat{\tau}(d,q)$-module and $\overline{V}$ is a finite-dimensional irreducible quotient of $V$, it follows that $\overline{V}$ is a finite-dimensional irreducible integrable module for $\tau(d,q).$ Thus $\overline{V}$ may be viewed as the non-graded module corresponding to the graded module $V$ we have fixed in the beginning.

Let $A_n = \mathbb{C}[t_1^{\pm 1}, \cdots , t_n^{\pm 1}]$ be the Laurent polynomial ring in $n$ commuting variables. We will now define an action of $\hat{\tau}(d,q)$ on the tensor product $\overline{V} \otimes A_n$.\\
For $m = (m_1, \cdots ,m_n) \in \mathbb{Z}^n$, any homogeneous element $x \otimes t^r \in \hat{\tau}(d,q)$ of degree $r$, and any $w \in \overline{V}$, define 
\[
(x \otimes t^r)\cdot (w \otimes t^m) = ((x \otimes t^r)\cdot w) \otimes t^{r+m},
\]
\[
d_i\cdot (w \otimes t^m) = m_i(w \otimes t^m), \quad 1 \le  i \le n.
\]
It is not difficult to check that these operations satisfy the defining Lie bracket relations of $\hat{\tau}(d,q)$, therefore they endow $\overline{V}\otimes A_n$ with the structure of a left $\hat{\tau}(d,q)$-module.

Since $V$ is $\mathbb{Z}^n$-graded as $\hat{\tau}(d,q)$-module,
\[
V = \bigoplus_{m \in \mathbb{Z}^n} V_m,
\]
the quotient map $\psi: V \to \overline{V}$ induces a natural map 
\begin{align*}
\Psi : V & \rightarrow \overline{V} \otimes A_n\\
v & \mapsto \bar{v}\otimes t^m,
\end{align*}
where $v \in V_m \; \text{for some} \; m \in \mathbb{Z}^n \; \text{and} \; \bar{v}= \psi(v)$.
It is easy to check that $\Psi$ is a nonzero homomorphism of the $\hat{\tau}(d,q)$-module. Therefore by the irreducibility of $V$, $\Psi$ is an injective homomorphism. Implies that $\Psi(V)$ is a nonzero irreducible $\hat{\tau}(d,q)$-submodule of $\overline{V} \otimes A_n$.

\begin{theorem}
\leavevmode
    \begin{enumerate}
        \item $\overline{V} \otimes A_n$ is a completely reducible integrable module for $\hat{\tau}(d,q)$ with respect to $(H \otimes \mathbb{C}_q) \oplus D$ with finite number of isomorphic irreducible components upto a grade shift.
        \item $V$ is isomorphic to an irreducible component of $\overline{V} \otimes A_n$ as the $\hat{\tau}(d,q)$-module.
    \end{enumerate}
\end{theorem}

\begin{proof}
    The proof of the theorem is identical to the proof of Proposition 3.8 and Theorem 3.9 of \cite{[13]}. There we need to assume the quantum torus $\mathbb{C}_q$ to be commutative, but the proofs are valid for any $\mathbb{C}_q$.
\end{proof}

\section{the final classification theorem}

The last section clearly reduces our study of irreducible integrable $\hat{\tau}(d,q)$ modules with finite-dimensional weight spaces to finite-dimensional irreducible $\tau(d,q)$ modules. These finite-dimensional irreducible modules have been completely classified in \cite{[18]}. In this Section , we recall the classification of $\tau(d,q)$-modules given in \cite{[18]} and henceforth classify all the irreducible integrable $\hat{\tau}(d,q)$-modules with finite-dimensional weight spaces. We will denote the set of all nonzero complex numbers by $\mathbb{C}^{\times}$.\\
Let $N_j$ be the least common multiple of orders of $q_{ij} \; (\text{i.e.} \; t_j^{N_j} \in Z(\mathbb{C}_q))$. Let us pick distinct scalars $a_{jk} \in \mathbb{C}^\times, \; 1 \leq k \leq M_j$ and define 
\[
Q_j(t_j) = \prod_{k=1}^{M_j} (t_j^{N_j} - a_{jk}), \quad \text{for} \quad 1 \le j \le n.
\]
Let $\mathcal{J}$ be the ideal of $\mathbb{C}_q$ generated by $Q_j(t_j)$'s and
\[
\tilde{\mathcal{J}} = \sum_{i \neq j} \mathcal{J} E_{ij} \; + \; \sum_{i \neq j} [\mathbb{C}_q E_{ij},\mathcal{J} E_{ji}], 
\]
then
\[
\tau(d,q)/\tilde{\mathcal{J}} \cong \bigoplus_{K-copies} \mathfrak{sl}_{dN}(\mathbb{C}),
\]
by extending
\[
X \otimes t^m \mapsto X \otimes \pi(t^m),
\]
where 
\[
\pi: \mathbb{C}_q \to \bigoplus_{K-copies}M_N(\mathbb{C})
\]
is the algebra homomorphism induced by quotienting by $\mathcal{J}$.
Let $(\rho,V)$ be an irreducible representation for $\bigoplus_{K}\mathfrak{sl}_{dN}(\mathbb{C})$, then $(\rho\circ\tilde{\pi},V)$ becomes a representation for $\tau(d,q)$ via a pull-back along
\[
\tau(d,q) \xrightarrow{\tilde{\pi}} \bigoplus_{K-copies} \mathfrak{sl}_{dN}(\mathbb{C}) \xrightarrow{\rho} \text{End}(V).
\]

\begin{theorem} {\normalfont(\cite[Theorem 6.6]{[18]})}\\
    Every finite-dimensional irreducible $\tau(d,q)$-module arises as the pullback of a finite-dimensional irreducible module for a direct sum of Lie algebras 
    \[
    \bigoplus_{K-copies} \mathfrak{sl}_{dN}(\mathbb{C})
    \]
    via the canonical surjective homomorphism of the Lie algebra
    \[
    \tilde{\pi}: \tau(d,q) \to \bigoplus_{K-copies} \mathfrak{sl}_{dN}(\mathbb{C}),
    \]
    where $N = N_1 N_2 \cdots N_n$, $K = M_1 M_2 \cdots M_n$ and each $M_j$ is the number of distinct roots of $Q_j$.
    
    In other words, every finite-dimensional irreducible representation of $\tau(d,q)$ is of the form  $(\rho \circ \tilde{\pi},V)$ for some finite-dimensional irreducible representation $(\rho,V)$ of 
    \[
    \bigoplus_{K-copies} \mathfrak{sl}_{dN}(\mathbb{C}).
    \]
\end{theorem}

\begin{theorem}\label{theorem 7.2}
    Let $V$ be an irreducible integrable $\hat{\tau}(d,q)$-module with finite-dimensional weight spaces and suppose that each central element $c_i \;(1 \leq i \leq n)$ acts trivially on V. Then $V$ is isomorphic to an irreducible component of $\overline{V} \otimes A_n$, where $\overline{V}$ is a finite-dimensional irreducible integrable $\tau(d,q)$-module obtained by pullback along the canonical homomorphism $\tilde{\pi}$, from some finite-dimensional irreducible integrable module for $\bigoplus_{K} \mathfrak{sl}_{dN}(\mathbb{C})$.
\end{theorem}

\medskip
Combining our results for the case of trivial central action (\cref{theorem 7.2}) with the classification obtained by S.~Eswara Rao and K.~Zhao~\cite{[18]} for modules with nontrivial central action, we arrive at a complete description of all irreducible integrable $\hat{\tau}(d,q)$-modules with finite-dimensional weight spaces. The following Theorem summarizes the full classification.
\medskip

\begin{theorem}\label{thm:7.3}
Let $V$ be an irreducible integrable $\hat{\tau}(d,q)$–module with finite-dimensional weight spaces. Then up to a grade shift, $V$ belongs to exactly one of the following two types

\begin{enumerate}
\item If some central element $c_i$ acts nontrivially on $V$, then $V$ is an irreducible component of
\[
\overline{W} \otimes A_{\,n-1}, \qquad \text{where} \quad
A_{\,n-1} = \mathbb{C}[t_1^{\pm1}, \dots, t_{\,n-1}^{\pm1}]
\]
and $\overline{W}$ is an irreducible integrable highest-weight module over 
\[
\tilde{\tau}^{(n)} = \tilde{\tau}(d,q) \oplus \mathbb{C} d_n.
\]
All irreducible components of $\overline{W} \otimes A_{\,n-1}$ are isomorphic up to a grade shift \cite[Theorem 6.12]{[18]}.

\item If all central elements $c_i, \; 1\le i \le n$ act trivially on $V$, then $V$ is an irreducible component of
\[
\overline{V} \otimes A_n, \qquad \text{where} \quad
A_n = \mathbb{C}[t_1^{\pm1}, \dots, t_n^{\pm1}]
\]
and $\overline{V}$ is a finite-dimensional irreducible integrable $\tau(d,q)$–module. 
All irreducible components of $\overline{V} \otimes A_n$ are isomorphic up to a grade shift
\normalfont {(\cref{theorem 7.2})}.
\end{enumerate}

\end{theorem}

\end{document}